\documentclass[final,1p,times,12pt]{elsarticle}
\usepackage{fullpage}
\usepackage[utf8]{inputenc}

\journal{Arxiv}

\makeatletter 
\def\ps@pprintTitle{%
 \let\@oddhead\@empty
 \let\@evenhead\@empty
 \def\@oddfoot{}%
 \let\@evenfoot\@oddfoot}
\makeatother

\usepackage{fullpage}
\usepackage[colorlinks]{hyperref}
\usepackage[utf8]{inputenc}

\biboptions{square,sort,comma,numbers}
\usepackage{graphicx}
\usepackage{amscd,amsmath,amsthm,amssymb}
\usepackage{verbatim}
\usepackage[all]{xy}
\usepackage{color}
\usepackage{hyperref}
\usepackage{tikz, float} \usetikzlibrary {positioning}

\usepackage{makeidx}
\usepackage{graphicx} 
\usepackage{multicol}
\usepackage{multirow}
\usepackage{array}
\usepackage{newtxmath} 
\makeindex

\newcommand{\BLambda}{{B_\ell}}

\def\wb{{\mathbf w}}
\def\Tc{{\mathcal T}}
\def\iff{\ifmmode\Longleftrightarrow \else
        \unskip${}\Longleftrightarrow{}$\ignorespaces\fi}
\let\:=\colon

\newtheorem{theorem}{Theorem}[section]
\newtheorem{lemma}[theorem]{Lemma}
\newtheorem{corollary}[theorem]{Corollary}

\theoremstyle{remark}
\newtheorem{remark}[theorem]{Remark}

\theoremstyle{definition}
\newtheorem{example}[theorem]{Example}

\newtheorem{definition}[theorem]{Definition}

\DeclareMathOperator{\Gr}{Gr}
\DeclareMathOperator{\Flag}{Fl}

\numberwithin{equation}{section}

\newcommand{\inwb}{{\rm in}_{{\bf w}_\ell}}
\newcommand{\inw}[1]{{\rm in}_{{\bf w}_{#1}}}

\newcommand{\II}{\mathcal{I}}

\DeclareMathOperator{\init}{in}

\begin{document}

\begin{frontmatter}

\title{Toric degenerations of flag varieties from matching field tableaux}


\author[add1]{Oliver Clarke}
\ead{oliver.clarke@bristol.ac.uk }
\address[add1]{
University of Bristol, School of Mathematics, BS8 1TW, Bristol, UK.}

\author[add2,add3]{Fatemeh Mohammadi}
\ead{fatemeh.mohammadi@ugent.be}
\address[add2]{
Department of Mathematics: Algebra and Geometry,
Ghent University, 
9000 Gent, 
Belgium}

\address[add3]{
Department of Mathematics and Statistics,
UiT – The Arctic University of Norway, 
9037 Troms\o, 
Norway}


\begin{abstract}
We present families of tableaux which interpolate between the classical semi-standard Young tableaux and matching field tableaux. Algebraically, this corresponds to SAGBI bases of Pl\"ucker algebras. We show that each such family of tableaux leads to a toric ideal, that can be realized as initial of the Pl\"ucker ideal, hence a toric degeneration for the flag variety. 
\end{abstract}
\end{frontmatter}

\section{Introduction}\label{sec:intro}
Computing toric degenerations of varieties is a standard tool to extend the deep relationship between combinatorics and toric varieties to more general spaces \cite{An13}. Given a projective variety, a toric degeneration is a flat family whose fibre over zero is a toric variety and all its other fibres are isomorphic to the original variety. Hence, toric degenerations enable us to extend the computational methods from toric varieties to general varieties.

\smallskip

In particular, toric degenerations of flag varieties are extensively studied, see e.g. \cite{FFL16, GL96}. The most known degeneration is the Gelfand-Cetlin degeneration, studied in \cite{FvectorGC, KOGAN}, which is related to semi-standard Young tableaux and their associated polytopes. This example has been generalized in different directions 
\cite{Feigin2012, irelli2019linear}.
On the other hand, the Gelfand-Cetlin degenerations naturally arise from the Gr\"obner basis theory and the SAGBI basis theory for Pl\"ucker algebras. More precisely, given an $n\times n$ matrix of indeterminants the Gelfand-Cetlin degeneration corresponds to a choice of weight vector which picks the diagonal term of each minor as its initial term. Therefore, a natural question is whether changing the weight vector would lead to other toric degnerations of flag varieties.

\smallskip

In this note, using the theory of matching fields \cite{sturmfels1996grobner} we introduce a family of tableaux as follows. Let $\Lambda(k,n)$ be a set of $k\times 1$ tableaux of integers corresponding to all  $k$-subsets of $[n]=\{1,\ldots,n\}$ and $\Lambda_n$, or $\Lambda$ when there is no confusion, be the union of $\Lambda(k,n)$'s. By combining tableaux from $\Lambda$ side by side, we can construct larger tableaux.
Then, we let $R = \mathbb{K}[P_I : I \in \Lambda]$ be the polynomial ring whose variables are corresponding to tableaux in $\Lambda$. Note that each tableau $T = \{I_1, \dots, I_t\}$ corresponds to a monomial $P^T = P_{I_1} \dots P_{I_t}$. We denote $J_{\Lambda}$ for the ideal generated by binomials $P^T - P^S$ where $S$ and $T$ are two row-wise equal tableaux. 
Our main motivation for studying $J_\Lambda$ is toric degenerations: Given a family of tableaux $\Lambda$ that are inferred from a matching field, the goal is to construct a weight vector ${\bf w}_\Lambda$ and estimate its corresponding initial Pl\"ucker ideal, also known as a Gr\"obner degeneration.
Our approach to calculating toric degenerations of flag varieties is to see whether $J_\Lambda$ can be realized as the initial of the Pl\"ucker ideal with respect to ${\bf w}_\Lambda$.
For a particular class called block diagonal matching fields, we show that this is indeed the case, and hence we provide a new family of toric degenerations for flag varieties.

\smallskip

Our work is related to computing the tropicalizations of Grassmannians and flag varieties \cite{bossinger2020families,speyer2004tropical, fink2015stiefel,KM16,KristinFatemeh, rietsch2017newton}. More precisely, we present a combinatorial method to compute several maximal cones in the tropicalizations of flag varieties.
We remark that for small flag varieties, calculating tropicalizations is explicitly carried out in \cite{bossinger2017computing}, however for larger dimensions theses computations are very hard.

\smallskip
The paper is structured as follows.
In \S\ref{sec:prelim}, we introduce the notation and definitions which we will use throughout the paper. In particular, we define matching fields and introduce their corresponding tableaux and ideals.
In \S\ref{sec:tableaux}, we highlight some important features of these tableaux and we show that the ideals $J_\Lambda$ are quadratically generated. In \S\ref{sec:flag}, we identify a family of $2$-column  matching field tableaux in bijection with $2$-column semi-standard tableaux, which enables us to transfer the classical algebraic results from semi-standard to matching field tableaux. 
We then apply our results to show that the set of Pl\"ucker forms is a SAGBI basis for the Pl\"ucker algebra with respect to the weight vector induced by any block diagonal matching field. This leads to a new family of toric degenerations of flag varieties.

\smallskip
\noindent{\bf Acknowledgement.} The authors are very grateful to J\"urgen Herzog
for helpful conversations. 
OC is supported by EPSRC Doctoral Training Partnership award EP/N509619/1.
FM is partially supported by a BOF Starting Grant of Ghent University and EPSRC Early Career Fellowship EP/R023379/1.


\section{Preliminaries}\label{sec:prelim}

Throughout we fix a field $\mathbb{K}$ with char$(\mathbb{K})=0$. We are mainly interested in the case when 
$\mathbb{K}=\mathbb{C}$. We let $[n]$ be the set $\{1, \dots, n \}$ and by $\mathbb{P}^n$ we denote the collection of all proper subsets of $[n]$. We let $S_n$ be the symmetric group on $[n]$. 

\smallskip

\noindent{\bf 1.1. Flag varieties.} 
A full flag is a sequence of vector subspaces of $\mathbb{K}^n$, 
$$\{0\}= V_0\subset V_1\subset\cdots\subset V_{n-1}\subset V_n=\mathbb{K}^n$$
where ${\rm dim}_{\mathbb{K}}(V_i) = i$. The set of all full flags is called the flag variety denoted by $\Flag_n$, which is naturally embedded  in a product of Grassmannians using the Pl\"ucker variables.
Each point in the flag variety can be represented by an $n\times n$ matrix $X=(x_{i,j})$ whose first $k$ rows generate $V_k$. Each $V_k$ corresponds to a point in the Grassmannian $\Gr(k,n)$. So the ideal of $\Flag_n$, denoted by $I_n$ is the kernel  of the map,
\[
\varphi_n:\  \mathbb{K}[P_J:\ \emptyset\neq J\subsetneq[n]]\rightarrow \mathbb{K}[x_{i,j}:\ 1\leq i\leq n-1,\ 1\leq j\leq n]
\]
sending each Pl\"ucker variable
$P_J$ to the determinant of the submatrix of $X$ with row indices $1,\ldots,|J|$ and column indices in $J$. 
Similarly, we define the Pl\"ucker ideal of $\Gr(k,n)$, denoted by $G_{k,n}$, as the kernel of the map $\varphi_n$, restricted to the variables $P_J$ with $|J|=k$.

\smallskip

\noindent{\bf 1.2. Gr\"obner degeneration.} 
Let $R=\mathbb{K}[{\bf P}]$ be the polynomial ring on $d$ Pl\"ucker variables and let $f=\sum a_{\textbf u} {\bf P}^{\textbf u}$ with $\textbf{u}\in \mathbb {Z}_{\ge 0}^{d}$. For each $\textbf w\in \mathbb{R}^{d}$ we define the \textit{initial form} of $f$ to be 
\[
\init_{\textbf w}(f)= \sum_{\textbf w\cdot \textbf u \text{ is minimal}} a_{\textbf u} {\bf P}^{\textbf u}.
\]
If $I$ is an ideal in $R$, then its \textit{initial ideal} with respect to $\textbf w$ is 
$
\init_{\textbf w}(I)=\langle \init _{\textbf w}(f) : f\in I\rangle.
$

An important feature of an initial ideal that we want to emphasize here is that it naturally gives rise to a flat degeneration of the variety $V(I)$ into the variety $V(\init_{\textbf w}(I))$ 
called the \textit{Gr\"obner degeneration}. 
Our goal is to find such degenerations where $\init_{\textbf w}(I)$ is toric. 

\medskip

\noindent {\bf 1.3. Matching fields.} 
A \emph{matching field} is a map $\Lambda_n : \mathbb{P}^n \rightarrow S_n$.
When there is no ambiguity, we write $\Lambda$ for $\Lambda_n$. Suppose $I = \{i_1 < \dots < i_k \} \subset [n]$,
we think of the permutation {$\sigma=\Lambda(I)$} as inducing a new ordering on the elements of $I$, 
where the position of $i_s$ is determined by the value of $\sigma(s)$. Note that we can restrict a matching field to all subsets of $[n]$ of size $k$ which are the matching fields associated to Grassmannians.

Given a matching field $\Lambda$ and a $k$-subset $I \subset [n]$, we represent the Pl\"ucker form $P_I$ as a $k \times 1$ tableau with entries given by $I$ ordered by $\Lambda$. Let $X=(x_{i,j})$ be an $n \times n$ matrix of indeterminants. To $I\subset [n]$ {with $\sigma=\Lambda(I)$} we associate the monomial 
$
\textbf{x}_{\Lambda(I)}:=x_{\sigma(1) i_{1}}x_{\sigma(2)i_2}\cdots x_{{\sigma(k)i_k}}. 
 $
A {\em matching field ideal} $J_\Lambda$ is defined as the kernel of the monomial map, 
\begin{eqnarray}\label{eqn:monomialmap}
\phi_{\Lambda} \colon\  & \mathbb{K}[P_I]  \rightarrow \mathbb{K}[x_{ij}]  
\quad\text{with}\quad
 P_{I}   \mapsto \text{sgn}(\Lambda(I)) \textbf{x}_{\Lambda(I)},
\end{eqnarray}
where sgn denotes the sign of the permutation $\Lambda(I)$. 
\smallskip

\begin{definition}
A matching field $\Lambda$  is \emph{coherent} if there exists an $n\times n$ matrix $M$ with entries in $\mathbb{R}$ 
such that for every proper subset $I$ in $\mathbb{P}^n$
the initial form of the  Pl\"ucker form  $P_I \in \mathbb{K}[x_{ij}]$,  $\text{in}_M (P_I) $ is $ \text{sgn}(\Lambda(I)) \mathbf{x}_{\Lambda(I)}$. Where, $\text{in}_M (P_I)$ is the sum of all terms in $P_I$ with the lowest weight with respect to $M$.
In this case, we say that the matrix $M$ \emph{induces the matching field} $\Lambda$.
\end{definition}

\begin{example}\label{example:coherent}
Consider the coherent matching field $\Lambda_4$ induced by the matrix $M$ below. The value $M_{ij}$ is the weight of the variable $x_{ij}$ which lies in the matrix $X$: 
\[
M = 
\begin{bmatrix}
0  &0  &0  &0  \\
3  &2  &1  &4  \\
8  &6  &4  &2  \\
12 &9  &6  &3  \\
\end{bmatrix}\, , \quad X = 
\begin{bmatrix}
x_{11}  &x_{12}  &x_{13}  &x_{14}  \\
x_{21}  &x_{22}  &x_{23}  &x_{24}  \\
x_{31}  &x_{32}  &x_{33}  &x_{34}  \\
x_{41}  &x_{42}  &x_{43}  &x_{44}  \\
\end{bmatrix}\, .
\]
Let us consider the Pl\"ucker form $P_{24}$ which is the determinant of the submatrix of $X$ consisting of the first two rows and the second and fourth columns.
\[
P_{24} =
\begin{vmatrix}
x_{12} & x_{14} \\
x_{22} & x_{24}
\end{vmatrix}
= x_{12}x_{24} - x_{14}x_{22}.
\]

The weight of the Pl\"ucker form is the minimum weight of monomials appearing in $P_{24}$. We see that the weight of $x_{12}x_{24}$ is $0 + 4 = 4$ and the weight of $-x_{14}x_{22}$ is $0 + 2 = 2$. So the weight of $P_{24}$ is $2$ and the initial term is $\init_{M}(P_{24}) = -x_{14}x_{22}$. We record $\init_{M}(P_{24})$ as a $2 \times 1$ tableau with entries $4,2$. Continuing in this way, we write down the single column tableaux arising from the matching field above:
\[
\begin{tabular}{|c|} \hline 1 \\ \hline \multicolumn{1}{c}{\,} \\ \multicolumn{1}{c}{\,} \\ \end{tabular}\ \  \begin{tabular}{|c|} \hline 2 \\ \hline \multicolumn{1}{c}{\,} \\ \multicolumn{1}{c}{\,} \\ \end{tabular}\ \ 
\begin{tabular}{|c|} \hline 3 \\ \hline \multicolumn{1}{c}{\,} \\ \multicolumn{1}{c}{\,} \\ \end{tabular}\ \ 
\begin{tabular}{|c|} \hline 4 \\ \hline \multicolumn{1}{c}{\,} \\ \multicolumn{1}{c}{\,} \\ \end{tabular}\ \quad
\begin{tabular}{|c|} \hline 1 \\ \hline 2 \\ \hline \multicolumn{1}{c}{\,} \\ \end{tabular}\ \  
\begin{tabular}{|c|} \hline 1 \\ \hline 3 \\ \hline \multicolumn{1}{c}{\,} \\ \end{tabular}\ \  
\begin{tabular}{|c|} \hline 4 \\ \hline 1 \\ \hline \multicolumn{1}{c}{\,} \\ \end{tabular}\ \  
\begin{tabular}{|c|} \hline 2 \\ \hline 3 \\ \hline \multicolumn{1}{c}{\,} \\ \end{tabular}\ \  
\begin{tabular}{|c|} \hline 4 \\ \hline 2 \\ \hline \multicolumn{1}{c}{\,} \\ \end{tabular}\ \  
\begin{tabular}{|c|} \hline 4 \\ \hline 3 \\ \hline \multicolumn{1}{c}{\,} \\ \end{tabular}\ \quad
\begin{tabular}{|c|} \hline 1 \\ \hline 2 \\ \hline 3 \\ \hline \end{tabular}\ \ 
\begin{tabular}{|c|} \hline 1 \\ \hline 2 \\ \hline 4 \\ \hline \end{tabular}\ \ 
\begin{tabular}{|c|} \hline 1 \\ \hline 3 \\ \hline 4 \\ \hline \end{tabular}\ \ 
\begin{tabular}{|c|} \hline 2 \\ \hline 3 \\ \hline 4 \\ \hline \end{tabular}\ .
\]

\end{example}

\begin{definition}\label{def:matching_field_weight}
Let $\Lambda$ be a coherent matching field induced by the matrix $M$. We define ${\bf w}_{\Lambda}$ to be the \emph{weight vector} induced by $M$ on the Pl\"ucker variables. That is, the component of ${\bf w}_{\Lambda}$ corresponding to the variable $P_I$ is the minimum weight of monomials appearing in $\phi_{\Lambda}$ with respect to $M$. The weight of a monomial is the sum of the corresponding terms in the weight matrix $M$. For ease of notation we write ${\bf P}^\alpha$ for the monomial $P_{I_1}^{\alpha_1} \dots P_{I_s}^{\alpha_s}$ where $\alpha = (\alpha_1, \dots, \alpha_s)$. And so the weight of ${\bf P}^\alpha$ is simply $\alpha \cdot {\bf w}_{\Lambda} $.
Given a coherent matching field $\Lambda$, we denote 
\emph{the initial ideal of $I_n$ with respect to $\wb_\Lambda$} by $\inw{\Lambda}(I_n)$.
\end{definition}

\begin{example}\label{example:diagonal_matching_field}
Consider the matching field $\Lambda_4$ from Example~\ref{example:coherent}. The Pl\"ucker variables are $P_1, P_2, P_3, P_4, P_{12}, P_{13}, P_{41}, P_{23}, P_{42}, P_{43}, P_{123}, P_{124}, P_{134}, P_{234}$ where, by convention, we order the entries of the indices according to initial terms of the corresponding Pl\"ucker form. The weight vector $\wb_{\Lambda_4}$ with the variables in this order is given by $$\wb_{\Lambda_4} = (0,0,0,0,2,1,3,1,2,1,6,4,3,3).$$ 
For instance, to obtain the weight of the Pl\"ucker variable $P_{124}$ we note that the initial term of the Pl\"ucker form is $\init{M}(P_{124}) = x_{11}x_{22}x_{34}$. We then sum the corresponding entries on $M$, i.e. the leading diagonal of the submatrix of $M$ consisting of the first three rows and the first, second and fourth columns. 

Computing the ideal $\inw{\Lambda_4}(I_4)$ in $\mathtt{Macaulay2}$ \cite{M2}, we obtain the following generating set,
\begin{multline*}
    \inw{\Lambda_4}(I_4) = \langle 
    P_{23} P_{134} - P_{13} P_{234}, 
    P_{43} P_{124} - P_{42} P_{134}, 
    P_{23} P_{124} - P_{12} P_{234},
    P_{13} P_{124} - P_{12} P_{134}, \\ 
    P_{13} P_{43}  - P_{12} P_{43}, 
    P_{2}  P_{134} - P_{1}  P_{234}, 
    P_{4}  P_{23} +  P_{2}  P_{43},
    P_{4}  P_{13} +  P_{1}  P_{43},
    P_{2}  P_{13} -  P_{1}  P_{23},
    P_{4}  P_{12} +  P_{1}  P_{42}
\rangle \, .
\end{multline*}

Let $\Lambda_3$ be the matching field $\Lambda_4$ restricted to the subsets of $\{1,2,3 \}$. The single column tableaux arising from $\Lambda_3$ are exactly those for $\Lambda_4$ which do not contain $4$. The weight vector is therefore $\wb_{\Lambda_3} = (0,0,0,2,1,1)$. 
The initial ideal with respect to $\wb_{\Lambda_3}$ is given by,
\[
\inw{\Lambda_3}(I_3) = \langle P_2 P_{13} - P_1 P_{23} \rangle = \inw{\Lambda_4}(I_4) \cap \mathbb{K}[P_I: \varnothing \neq I \subsetneq \{1,2,3 \}].
\]

Note that for $\Lambda_3$, the entries in each column of each tableau are strictly increasing. We call such matching field \emph{diagonal}. For diagonal matching fields, their corresponding Gr\"obner cones are related to Gelfand-Cetlin polytopes.
\end{example}

\noindent{\bf 1.4. Block diagonal matching fields.} 
Given $n$ and $0\leq\ell<n$, we define the block diagonal matching field denoted by $B_\ell=(1\cdots \ell|\ell+1\cdots n)$
as the map $B_\ell:\mathbb{P}^n\rightarrow S_n$ with
\[
 B_\ell(I)= \left\{
     \begin{array}{@{}l@{\thinspace}l}
      id  &: \text{if $\lvert I|=1$ or $\lvert I \cap \{1,\ldots,\ell\}\rvert \ge 2$},\\
      (12)  &: \text{otherwise}. \\
     \end{array}
   \right.
\]
These matching fields are called $2$-block diagonal in \cite{KristinFatemeh}. Each matching field $B_\ell$ is induced by the weight matrix:
\[
M_{\ell}=\begin{bmatrix}
    0   & 0 & \cdots & 0&  0& 0& \cdots & 0  \\
\ell&\ell-1&\cdots&1&n&n-1&\cdots&\ell+1\\
    2n   & 2(n-1) & \cdots & 10&  8& 6&4 & 2  \\
    \cdots  & \cdots  & \cdots &  \cdots & \cdots  & \cdots &  \cdots  &  \cdots   \\
     (n-1)n& (n-1)^2& \cdots   & 5(n-1)& 4(n-1)  & 3(n-1) & 2(n-1)  & n-1    \\
\end{bmatrix}.
\]
Therefore, all block diagonal matching fields are coherent. We note that if $\ell = 0$ or $\ell = n$ then the matching field is diagonal, see Example~\ref{example:diagonal_matching_field}. 
We denote by $\wb_{\ell}$ the weight vector $\wb_{B_{\ell}}$, see Definition~\ref{def:matching_field_weight}.

\begin{example}
For $n = 3$, we calculate the weight vector $\wb_{\ell}$ and matching field ideals for each block diagonal matching field $B_{\ell}$. The Pl\"ucker ideal is given by $I_3 = \langle P_1 P_{23} - P_2 P_{13} + P_3 P_{12} \rangle $. The weight matrices are given by,
\[
M_0 = 
\begin{bmatrix}
0  &0  &0  \\
3  &2  &1  \\
6  &4  &2  \\
\end{bmatrix} \, ,
\quad 
M_1 =
\begin{bmatrix}
0  &0  &0  \\
1  &3  &2  \\
6  &4  &2  \\
\end{bmatrix} \, ,
\quad 
M_2 = 
\begin{bmatrix}
0  &0  &0  \\
2  &1  &3  \\
6  &4  &2  \\
\end{bmatrix} \, .
\quad 
\]

For each matching field we can write down the weight vector and initial ideal,

\[
 \resizebox{8cm}{!}{\begin{tabular}{|c|c|c|}
    \hline
    $\ell$  & \begin{tabular}{c} Weight vector $\wb_{\ell}$ for \\ $P_1, P_2, P_3, P_{12}, P_{13}, P_{23}$   \end{tabular} & Ideal $\inwb(I_3)$ \\
    \hline
    0 & $(0,0,0,2,1,1)$ & $\langle P_1 P_{23} - P_2 P_{13} \rangle$ \\
    1   & $(0,0,0,1,1,2)$ & $\langle P_3 P_{12} - P_2 P_{13} \rangle$ \\
    2   & $(0,0,0,1,2,1)$ & $\langle P_1 P_{23} + P_3 P_{12} \rangle$ \\
    \hline
\end{tabular}\, .
}\]

\end{example}

\section{Matching field ideals}\label{sec:tableaux}
In this section, we examine a generating set of block diagonal matching field ideals. In particular, using the combinatorics of matching field tableaux we show that all these ideals are quadratically generated. Notice that, for a general ideal, providing a bound for the degree of its generators is a difficult problem. Such questions are usually studied for  special classes of ideals which are define combinatorially, see e.g.~\cite{White, hibi1987distributive, Ohsugi, ene2011monomial, Cone, mohammadi2010weakly}. 
\begin{figure}
    \centering
    $T = $
    \begin{tabular}{c|c|c|c|c|c}
        \hline
        \multicolumn{1}{|c}{1}  & \multicolumn{1}{|c|}{4}    & 2                         & 5                        & 2                          & \multicolumn{1}{c|}{4}\\
        \cline{1-6}
        \multicolumn{1}{c}{\,}  & \multicolumn{1}{|c|}{1}    & 3                         & 2                        & 3                          & \multicolumn{1}{c|}{2}\\
        \cline{2-6}
        \multicolumn{1}{c}{\,}  & \multicolumn{1}{|c|}{6}    & \multicolumn{1}{c|}{5}    & \multicolumn{1}{c|}{\,}  & \multicolumn{1}{c|}{4}    & \multicolumn{1}{c|}{5} \\
        \cline{2-3} \cline{5-6}
        \multicolumn{1}{c}{\,}  & \multicolumn{1}{c|}{\,}     & \multicolumn{1}{c|}{6}   & \multicolumn{1}{c|}{\,}  & \multicolumn{1}{c|}{5}    & \multicolumn{1}{c|}{6} \\
        \cline{3-3} \cline{5-6}
    \end{tabular}\, ,
    \quad $T' = $
    \begin{tabular}{c|c|c|c||c|c}
        \multicolumn{4}{c}{$X$} & \multicolumn{2}{c}{$Y$}\\
        \hline
        \multicolumn{1}{|c|}{1}  & 2                         & 2                         & 5                        & 4                         & \multicolumn{1}{c|}{4}  \\
        \cline{1-6}
        \multicolumn{1}{|c|}{2}  & 3                         & \multicolumn{1}{c|}{3}    & \multicolumn{1}{c|}{\,}  & 1                         & \multicolumn{1}{c|}{2}  \\
        \cline{1-3} \cline{5-6}
        \multicolumn{1}{|c|}{4}  & \multicolumn{1}{c|}{5}    & \multicolumn{1}{c}{\,}    & \multicolumn{1}{c|}{\,}  & 5                         & \multicolumn{1}{c|}{6}  \\
        \cline{1-2} \cline{5-6}
        \multicolumn{1}{|c|}{5}  & \multicolumn{1}{c|}{6}    & \multicolumn{1}{c}{\,}    & \multicolumn{1}{c|}{\,}  & \multicolumn{1}{c|}{6}    &                         \\
        \cline{1-2} \cline{5-5}
    \end{tabular}\, ,
    \quad $T'' = $
    \begin{tabular}{c|c|c|c|c|c}
        \hline
        \multicolumn{1}{|c|}{1}  & 1                         & 1                         & 4                         & 5                         & \multicolumn{1}{c|}{6}  \\
        \cline{1-6}
        \multicolumn{1}{|c|}{2}  & 3                         & 3                         & \multicolumn{1}{c|}{6}    & \multicolumn{1}{c}{\,}    &   \\
        \cline{1-4}
        \multicolumn{1}{|c|}{3}  & 4                         & \multicolumn{1}{c|}{5}    & \multicolumn{1}{c}{\,}    & \multicolumn{1}{c}{\,}    &   \\
        \cline{1-3}
        \multicolumn{1}{|c|}{5}  & \multicolumn{1}{c}{\,}    & \multicolumn{1}{c}{\,}    & \multicolumn{1}{c}{\,}    & \multicolumn{1}{c}{\,}    &   \\
        \cline{1-1} 
    \end{tabular}
    \caption{Three examples of tableaux for $\Flag_6$ and block diagonal matching field $B = (123|456)$. The tableau $T''$ is in semi-standard form as the entries of each row are weakly increasing and the entries of each column are strictly increasing. The tableaux $T$ and $T'$ are row-wise equal and $T'$ has been partitioned into two parts $X$ and $Y$. For each column $I \in X$ we have $B(I) = id$ and for each $I \in Y$ we have $B(I) = (12)$. Note that $X$ and $Y$ have weakly increasing rows, in particular $X$ is in semi-standard form.}
    \label{fig:tableau_exmple}
\end{figure}

\begin{definition}\label{def:swap}
 Suppose that $T$ and $T'$ are two tableaux that are row-wise equal. 
We say that $T$ and $T'$ differ by a {\em swap}, or $T'$ is obtained from $T$ by a {\em quadratic relation}, if $T$ and $T'$ are the same for all but two columns. If two tableaux differ by a sequence of swaps then we say that the tableaux are {\em quadratically equivalent}.
\end{definition}

Given an arbitrary matching field $\Lambda$ and two row-wise equal tableaux we cannot guarantee that the tableaux are quadratically equivalent, see Example~\ref{example:not_quad_equiv}.
However in the case of block diagonal matching fields, row-wise equal does imply quadratically equivalent. So for example in Figure~\ref{fig:tableau_exmple}, it is possible to find a sequence of quadratic relations from $T$ to $T'$.

\begin{example}\label{example:not_quad_equiv}
Fix $n=6$ and suppose we have a matching field $\Lambda$ such that the only subsets $I$ of $\{1, \dots, 6 \}$ for which $\Lambda(I) = id$ are those which appear in $T$ and $T'$ where,
\[
T = 
\begin{tabular}{|c|c|c|}
    \hline
    1 & 2 & 3  \\
    \hline
    4 & 5 & 6  \\
    \hline
\end{tabular}\, ,
\quad T' = 
\begin{tabular}{|c|c|c|}
    \hline
    3 & 1 & 2  \\
    \hline
    4 & 5 & 6  \\
    \hline
\end{tabular}
\ .
\]
The tableaux $T$ and $T'$ are row-wise equal, however it is not possible to apply a quadratic relation to either tableau. Hence $T$ and $T'$ are not quadratically equivalent.
\end{example}

\begin{theorem}\label{prop:quad}
The ideals of block diagonal matching fields are  quadratically generated.
\end{theorem}

\begin{proof}
Suppose that $T$ and $T'$ are row-wise equal tableaux.
The first step is to show that there is a sequence of quadratic relations after which $T$ and $T'$ have the same columns of size $1$. We define $P_T$ to be the sub-tableau of $T$ which contains all columns of $T$ but only the first two rows. Similarly we define $P_{T'}$. We form a sub-tableau $\overline{T}$ of $T$ and $\overline{T'}$ of $T'$ as follows. Initially set $\overline{T} = T$ and $\overline{T'} = T'$. For each column $C$ in $P_T$, we check whether there is a column $C'$ in $P_{T'}$ with the same entries. If $C'$ exists, then we remove the column in $\overline{T}$ corresponding to $C$ and remove the column in $\overline{T'}$ corresponding to $C'$. If $\overline{T}$ and $\overline{T'}$ are empty, then the columns of size $1$ in $T$ and $T'$ are equal. 

Assuming that the columns of size $1$ in $T$ and $T'$ are different, we will show that there is always a swap which reduces the size of $\overline{T}$.

We write $a_1 \le \dots \le a_t$ for the entries of the columns of size $1$ in $\overline{T}$ and $b_1 \le \dots \le b_t$ for the entries of the columns of size $1$ in $\overline{T'}$. Suppose that $s$ is the largest index such that $a_s\neq b_s$. Without loss of generality we assume that $a_s < b_s$. Suppose that $a_s, b_s \in B_2$, then there is a column $C = (b_s, c_2, \dots, c_r)^{tr}$ in $\overline{T}$. We may swap $a_s$ and $b_s$ in $T$. Now there is a new column $(b_s)$ of size one in $T$ which appears in $T'$. So we have reduced the number of size one columns in $\overline{T}$ and $\overline{T'}$.
We may now assume that $a_s \in B_1$. Since the first row of $\overline{T}$ and $\overline{T'}$ are equal as multisets, $a_s$ appears in a column $C = (a_s,c_2,c_3 \dots, c_r)^{tr}$ in $\overline{T'}$, where $r \ge 2$. Since $a_s \in B_1$ it follows that $c_2 \in B_1$. Since the second rows of $\overline{T}$ and $\overline{T'}$ are equal as multisets, it follows that there is a column $D = (d_1, c_2, d_3, \dots, d_t)^{tr}$ in $\overline{T}$, for some $t \ge 2$. By assumption we have that $C \in \overline{T}$ and $D \in \overline{T}'$. Therefore $d_1 \neq a_s$. We may swap $a_s$ and $d_1$ in $T$. Since the pair $(a_s, c_2)^{tr}$ appears in $P_{T'}$ we have reduced the number of columns in $\overline{T}$ and $\overline{T'}$. Since there are finitely many columns in $\overline{T}$, it follows that there is a sequence of swaps which makes the columns of size one in $T$ and $T'$ equal.

To complete the proof, by the first step, we may assume that $T$ and $T'$ do not contain any columns of size one. We now construct tableaux $S$ and $S'$ by filling the columns of $T$ and $T'$ which have size less than $n$ as follows. Suppose we have an empty space in $T$ and $T'$ in row $i$, then we fill this space with the entry $n+i$. We do this for all spaces to obtain a rectangular tableau with $n$ rows whose entries lie in $\{1,\ldots,2n\}$. We have that $S$ and $S'$ are valid tableaux for the block diagonal matching field $B' = (B_1, B_2 \cup \{n+1, \dots, 2n \})$. Since $S$ and $S'$ are row-wise equal, by the Grassmannian case \cite[Theorem 4.1]{OllieFatemeh}, there is a sequence $S = S_1, S_2, \dots, S_m = S'$  of quadratic relations from $S$ to $S'$. For each $i$, we remove the entries in the tableau $S_i$ lying in $\{n+1, \dots, 2n \}$ to obtain a sequence of quadratic relation from $T$ to $T'$, which completes the proof.
\end{proof}


\section{Matching field algebras}\label{sec:flag}
In this section, we show that the set of Pl\"ucker forms is a SAGBI basis for the Pl\"ucker algebra with respect to the weight vectors induced by block diagonal matching fields. More precisely, we determine a family of $2$-column matching field tableaux which are in bijection with the set of $2$-column semi-standard tableaux. Note that a tableau is in {\em semi-standard form} if the entries of each row are weakly increasing and the entries of each column are strictly increasing, see for example Figure~\ref{fig:tableau_exmple}.

First, we recall the definition of SAGBI basis
from \cite{robbiano1990subalgebra} in our setting. 

\begin{definition}\label{sagbi}
Let $\mathcal{A}_{n}$ be the Pl\"ucker algebra of $\Flag_n$ and let $A_{\ell} = \mathbb{K}[\init_{M_\ell}(P_I) : \varnothing \neq I \subsetneq [n]]$ be the \emph{algebra} of $B_\ell$. The set of Pl\"ucker forms $\{P_I: \varnothing \neq I\subsetneq [n]\} \subset \mathbb{K}[x_{ij}]$ is a SAGBI basis for $\mathcal{A}_{n}$ with respect to the weight vector $\wb_\ell$ if and only if for each $I \subset [n]$, the initial form $\inwb(P_I) $ is a monomial and
$\inwb(\mathcal{A}_{n})=A_\ell$. 

Similarly, denoting $\mathcal{A}^k_{n}$ and $A_{\ell}^k$ for the Pl\"ucker algebra of $\Gr(k,n)$ and the algebra of $B_\ell$ restricted to Pl\"ucker variables of size $k$, we say that the set $\{P_I: |I|=k\}$ is a SAGBI basis for $\mathcal{A}^k_{n}$ if and only if the initial forms $\inwb(P_I)$ are monomials and
$\inwb(\mathcal{A}^k_{n})=A_\ell^k$. 
\end{definition}

\noindent{\bf 4.1. Initial algebras of $\Gr(k,n)$.} 
First we recall the following definition from \cite{OllieFatemeh}.

\begin{definition}\label{def:basis_gr}
Fix $\Gr(k,n)$ and $B_{\ell} = (B_1 | B_2)$ a matching field. Let $\Tc^k_{\ell}$ be the collection of all tableaux $T$ which follow. We partition $\Tc^k_{\ell}$ into types, each of which is described below. We also define a map $S^k$ taking each tableau $T \in \Tc^k_{\ell}$ to a semi-standard tableau $S^k(T)$. Note that $S^k(T)$ does not necessarily lie in $\Tc^k_{\ell}$. We write,
\[
I = \{i_1 < i_2 < \dots < i_k\} \text{ and } 
J = \{j_1 < j_2 < \dots < j_k\}.
\]

\textbf{Type 1.}
\[
T = 
\begin{tabular}{|c|c|}
    \hline
    $i_1$     & $j_1$       \\ \hline
    $i_2$     & $j_2$       \\ \hline
    $\vdots$  & $\vdots$    \\ \hline
    $i_k$     & $j_k$       \\
    \hline
\end{tabular} \quad
\begin{tabular}{l}
    Where, \\
    $i_1 \le j_1, i_2 \le j_2, \dots, i_k \le j_k$.
\end{tabular} \quad
S^k(T) = 
\begin{tabular}{|c|c|}
    \hline
    $i_1$     & $j_1$       \\ \hline
    $i_2$     & $j_2$       \\ \hline
    $\vdots$  & $\vdots$    \\ \hline
    $i_k$     & $j_k$       \\
    \hline
\end{tabular}\ .
\]

\medskip

\textbf{Type 2.}
\[
T = 
\begin{tabular}{|c|c|}
    \hline
    $i_2$     & $j_2$       \\ \hline
    $i_1$     & $j_1$       \\ \hline
    $i_3$     & $j_3$       \\ \hline
    $\vdots$  & $\vdots$    \\ \hline
    $i_k$     & $j_k$       \\ 
    \hline
\end{tabular} \quad
\begin{tabular}{l}
    Where, \\
    $i_1 \le j_1, i_2 \le j_2, \dots, i_k \le j_k$.
\end{tabular}\quad
S^k(T) = 
\begin{tabular}{|c|c|}
    \hline
    $i_1$     & $j_1$       \\ \hline
    $i_2$     & $j_2$       \\ \hline
    $i_3$     & $j_3$       \\ \hline
    $\vdots$  & $\vdots$    \\ \hline
    $i_k$     & $j_k$       \\
    \hline
\end{tabular}.
\]

\medskip

\textbf{Type 3A.}
\[
T =
\begin{tabular}{|c|c|}
    \hline
    $i_2$     & $j_1$       \\ \hline
    \textcolor{blue}{$i_1$}     & $j_2$       \\ \hline
    $i_3$     & $j_3$       \\ \hline
    $\vdots$  & $\vdots$    \\ \hline
    $i_k$     & $j_k$       \\
    \hline
\end{tabular} \quad
\begin{tabular}{l}
    Where, \\
    \textcolor{blue}{$i_1 \in B_1$},\\
    $i_2, \dots, i_k, j_1 \dots, j_k \in B_2$,\\
    $i_2 \le j_1$,\\
    $i_3 \le j_3, \dots, i_k \le j_k$.
\end{tabular} \quad
S^k(T) = 
\begin{tabular}{|c|c|}
    \hline
    \textcolor{blue}{$i_1$}     & $j_1$       \\ \hline
    $i_2$     & $j_2$       \\ \hline
    $i_3$     & $j_3$       \\ \hline
    $\vdots$  & $\vdots$    \\ \hline
    $i_k$     & $j_k$       \\
    \hline
\end{tabular}\ .
\]

\medskip

\textbf{Type 3B(r).}\label{page:3B}
\[
T = 
\begin{tabular}{|c|c|}
    \hline
    $i_2$           & $j_1$       \\ \hline
    \textcolor{blue}{$i_1$}           & $j_2$       \\ \hline
    $i_3$           & $j_3$       \\ \hline
    $\vdots$        & $\vdots$    \\ \hline
    $i_{r-1}$       & $j_{r-1}$   \\ \hline
    $i_{r}$         & $j_{r}$     \\ \hline
    $i_{r+1}$       & $j_{r+1}$   \\ \hline
    $\vdots$        & $\vdots$    \\ \hline
    $i_k$           & $j_k$       \\
    \hline
\end{tabular} \quad
\begin{tabular}{l}
    Where, \\
    \textcolor{blue}{$i_1 \in B_1$}, \\
    $i_2, \dots, i_k, j_1 \dots, j_k \in B_2$,\\
    $j_1 < j_2 \le i_2$,\\
    $r = \min\{t \ge 2 :j_{t+1} > i_{t}\}$,\\
    $i_3 > j_3, \dots, i_r > j_r$,  \\
    $i_{r+1} \le j_{r+1}, \dots, i_k \le j_k$. \\
\end{tabular}\quad
S^k(T) = 
\begin{tabular}{|c|c|}
    \hline
    \textcolor{blue}{$i_1$}           & $j_1$       \\ \hline
    $j_2$           & $i_2$       \\ \hline
    $j_3$           & $i_3$       \\ \hline
    $\vdots$        & $\vdots$    \\ \hline
    $j_{r-1}$       & $i_{r-1}$   \\ \hline
    $j_{r}$         & $i_{r}$     \\ \hline
    $i_{r+1}$       & $j_{r+1}$   \\ \hline
    $\vdots$        & $\vdots$    \\ \hline
    $i_k$           & $j_k$       \\
    \hline
\end{tabular}\ .
\]

\medskip

\textbf{Type 3C(s).}\label{page:3C}
\[
T = 
\begin{tabular}{|c|c|}
    \hline
    $i_2$           & \textcolor{blue}{$j_1$}       \\ \hline
    \textcolor{blue}{$i_1$}           & \textcolor{blue}{$j_2$}       \\ \hline
    $i_3$           & \textcolor{blue}{$j_3$}       \\ \hline
    $\vdots$        & \textcolor{blue}{$\vdots$}    \\ \hline
    $i_{s}$         & \textcolor{blue}{$j_{s}$}     \\ \hline
    $i_{s+1}$       & $j_{s+1}$   \\ \hline
    $\vdots$        & $\vdots$    \\ \hline
    $i_k$           & $j_k$       \\
    \hline
\end{tabular} \quad
\begin{tabular}{l}
    Where, \\
    \textcolor{blue}{$i_1, j_1, j_2, \dots, j_s \in B_1$},\\
    $i_2, \dots, i_k, j_{s+1} \dots, j_k \in B_2$,\\
    $i_1 \le j_1 < j_2$,\\
    $i_{s+1} \le j_{s+1}, \dots, i_k \le j_k$.
\end{tabular} \quad
S^k(T) = 
\begin{tabular}{|c|c|}
    \hline
    \textcolor{blue}{$i_1$}           & \textcolor{blue}{$j_1$}       \\ \hline
    \textcolor{blue}{$j_2$}           & $i_2$       \\ \hline
    \textcolor{blue}{$j_3$}           & $i_3$       \\ \hline
    \textcolor{blue}{$\vdots$}        & $\vdots$    \\ \hline
    \textcolor{blue}{$j_{s}$}         & $i_{s}$     \\ \hline
    $i_{s+1}$       & $j_{s+1}$   \\ \hline
    $\vdots$        & $\vdots$    \\ \hline
    $i_k$           & $j_k$       \\
    \hline
\end{tabular}\ .
\]

\medskip

\textbf{Type 3D(s).}\label{page:3D}
\[
T = 
\begin{tabular}{|c|c|}
    \hline
    $i_2$           & \textcolor{blue}{$j_1$}       \\ \hline
    \textcolor{blue}{$i_1$}           & \textcolor{blue}{$j_2$}       \\ \hline
    $i_3$           & \textcolor{blue}{$j_3$}       \\ \hline
    $\vdots$        & \textcolor{blue}{$\vdots$}    \\ \hline
    $i_{s}$         & \textcolor{blue}{$j_{s}$}     \\ \hline
    $i_{s+1}$       & $j_{s+1}$   \\ \hline
    $\vdots$        & $\vdots$    \\ \hline
    $i_k$           & $j_k$       \\
    \hline
\end{tabular} \quad
\begin{tabular}{l}
    Where, \\
    \textcolor{blue}{$i_1, j_1, j_2, \dots, j_s \in B_1$},\\
    $i_2, \dots, i_k, j_{s+1} \dots, j_k \in B_2$,\\
    $i_1 \ge j_2$,\\
    $i_{s+1} \le j_{s+1}, \dots, i_k \le j_k$.
\end{tabular} \quad
S^k(T) = 
\begin{tabular}{|c|c|}
    \hline
    \textcolor{blue}{$i_1$}           & \textcolor{blue}{$j_1$}       \\ \hline
    \textcolor{blue}{$j_2$}           & $i_2$       \\ \hline
    \textcolor{blue}{$j_3$}           & $i_3$       \\ \hline
    \textcolor{blue}{$\vdots$}        & $\vdots$    \\ \hline
    \textcolor{blue}{$j_{s}$}         & $i_{s}$     \\ \hline
    $i_{s+1}$       & $j_{s+1}$   \\ \hline
    $\vdots$        & $\vdots$    \\ \hline
    $i_k$           & $j_k$       \\
    \hline
\end{tabular}\ .
\]
\end{definition}

The following theorem shows that the collection of tableaux in Definition~\ref{def:basis_gr} forms a SAGBI basis for the Pl\"ucker algebra, see \cite[Theorem~4.3]{OllieFatemeh}.
\begin{theorem}\label{thm:basis_gr}
For each pair of integers $k$ and $\ell$, we have that $\inwb(\mathcal{A}_{n}^k)=A_{\ell}^k$. Moreover, following the notation of Definition~\ref{def:basis_gr}, the set $\Tc^k_{\ell}$ of tableaux
is a basis for the degree two subspace of $A_{\ell}^k$ denoted by $[A_{\ell}^k]_2$, and the map $S^k$ is a bijection.
\end{theorem}

\noindent{\bf 4.2. Initial algebras of $\Flag_n$.} 

\smallskip 

Extending our results for Grassmannians, we now introduce a family of $2$-column matching field tableaux which are in bijection with the set of $2$-column semi-standard tableaux. 

\medskip

\noindent{\bf Notation.} 
Fix a block diagonal matching field $B_\ell=(1\cdots \ell|\ell+1\cdots n)$. For each collection $\II = \{I_1, \dots, I_r\}$ of non-empty subsets $ I_i \subset [n]$, we denote by $T_{\II}$, or when there are few columns by $T_{I_1, \dots, I_r}$, the tableau with columns $I_i$ for $1 \le i \le r$. The order of elements in column $I_i$ are given by the matching field $B_\ell$.
\begin{definition}\label{def:basis_fl}
Fix $B_{\ell} = (B_1 | B_2)$ a matching field. We define $\Tc_{\ell}$ to be the collection of all tableau $T$ which follow. We also describe the map $S$ which assigns to each $T $ in $ \Tc_{\ell}$ a semi-standard tableau $S(T)$. To help identify tableaux, we assign to each tableau in $\Tc_{\ell}$ a type. We proceed by listing all tableaux $T$ in $\Tc_{\ell}$ by type along with the image of $T$ under $S$. 

By convention for any subsets $I,J,I',J' \subset [n]$, we write,
\[
I = \{i_1 < i_2 < \dots < i_s\}, \quad
J = \{j_1 < j_2 < \dots < j_t\},
\]
\[
I' = \{i'_1 < i'_2 < \dots < i'_{s'}\}, \quad
J' = \{j'_1 < j'_2 < \dots < j'_{t'}\}.
\]

\textbf{Type 1.} For each $1 \le k \le n$ and $T_{IJ}$ in $ \Tc^k_{\ell}$, let $T = T_{IJ}$ and define $S(T) = S^k(T)$. In this case we have $s = t = k$. We subdivide this type of tableau into further types. 
We say that $T $ in $ \Tc_{\ell}$ is of type 1.$X$, where 
$X \in \{1,2,3A,3B(r),3C(s),3D(s): 2 \le r,s \le k \}$ 
if $T$ is of type $X$ in $\Tc^k_{\ell}$ according to Definition~\ref{def:basis_gr}. So, for instance, if $T$ in $ \Tc_{\ell}^k$ is of type $3B(2)$, then we say $T$ is of type $1.3B(2)$ in $\Tc_{\ell}$.
\smallskip

\textbf{Type 2.} Let $s > t \ge 2$ and $T_{I'J'} $ in $ \Tc^{t}_{\ell}$. In this case we have $s' = t' = t$. Suppose $T_{I'J'}$ is of type $X$ in $\Tc^t_{\ell}$ as in Definition~\ref{def:basis_gr}. Let $m = \min\{i'_t, j'_t\}$ and pick any subset $\{i_{t+1}, \dots, i_s \} \subset [n]$ with $m < i_{t+1} < \dots < i_s \le n$. If $i'_t \le j'_t$ then let,
\[
I = I' \cup \{i_{t+1}, \dots, i_s \}, \quad J = J'.
\]
Otherwise if $i'_t > j'_t$ then let,
\[
I = I', \quad J = J' \cup \{i_{t+1}, \dots, i_s \}.
\]
Then let $T = T_{IJ}$. We say that $T$ is of type 2.$X$. Note that $i'_t > j'_t$ only if $T$ is of type 2.3B(t), 2.3C(t) or 2.3D(t). Let $T_{I'' J''} = S^t(T_{I'J'})$ be a tableau in semi-standard form. We set 
\[
S(T) = T_{I'' \cup \{i_{t+1}, \dots, i_{s} \} J''}.
\]

\smallskip

\textbf{Type 3.} Fix $s \ge 2$ and $t = 1$. We partition this type into the following sub-types. 

\textbf{Type 3A(r).} Where $r \in \{1,2\}$. 
\[
T = 
\begin{tabular}{|c|c|}
    \cline{1-2}
    $i_1$ & $j_1$ \\
    \cline{1-2}
    $i_2$ \\
    \cline{1-1}
    $i_3$ \\
    \cline{1-1}
    $\vdots$ \\
    \cline{1-1}
    $i_s$ \\
    \cline{1-1}
\end{tabular},
\quad
\begin{tabular}{l}
    $i_1, i_2 \in B_r$, \\
    $i_1 \le j_1$.
\end{tabular}
\quad
S(T) = T.
\]

\textbf{Type 3B.}
\[
T = 
\begin{tabular}{|c|c|}
    \cline{1-2}
    $i_1$ & $j_1$ \\
    \cline{1-2}
    $i_2$ \\
    \cline{1-1}
    $i_3$ \\
    \cline{1-1}
    $\vdots$ \\
    \cline{1-1}
    $i_s$ \\
    \cline{1-1}
\end{tabular},
\quad
\begin{tabular}{l}
    $i_1, i_2 \in B_2$,\\
    $j_1 \in B_1$.
\end{tabular}
\quad
S(T) = 
\begin{tabular}{|c|c|}
    \cline{1-2}
    $j_1$ & $i_1$ \\
    \cline{1-2}
    $i_2$ \\
    \cline{1-1}
    $i_3$ \\
    \cline{1-1}
    $\vdots$ \\
    \cline{1-1}
    $i_s$ \\
    \cline{1-1}
\end{tabular}.
\]

\textbf{Type 3C.}
\[
T = 
\begin{tabular}{|c|c|}
    \cline{1-2}
    $i_2$ & $j_1$ \\
    \cline{1-2}
    $i_1$ \\
    \cline{1-1}
    $i_3$ \\
    \cline{1-1}
    $\vdots$ \\
    \cline{1-1}
    $i_s$ \\
    \cline{1-1}
\end{tabular},
\quad
\begin{tabular}{l}
    $i_1 \in B_1$,\\
    $j_1, i_2, \dots, i_s \in B_2$,\\
    $i_2 \le j_1$.
\end{tabular}
\quad
S(T) = 
\begin{tabular}{|c|c|}
    \cline{1-2}
    $i_1$ & $j_1$ \\
    \cline{1-2}
    $i_2$ \\
    \cline{1-1}
    $i_3$ \\
    \cline{1-1}
    $\vdots$ \\
    \cline{1-1}
    $i_s$ \\
    \cline{1-1}
\end{tabular}.
\]

\textbf{Type 3D.}
\[
T = 
\begin{tabular}{|c|c|}
    \cline{1-2}
    $i_2$ & $j_1$ \\
    \cline{1-2}
    $i_1$ \\
    \cline{1-1}
    $i_3$ \\
    \cline{1-1}
    $\vdots$ \\
    \cline{1-1}
    $i_s$ \\
    \cline{1-1}
\end{tabular},
\quad
\begin{tabular}{l}
    $i_1, j_1 \in B_1$,\\
    $i_2, \dots, i_s \in B_2$,\\
    $i_1 \le j_1$.
\end{tabular}
\quad
S(T) = 
\begin{tabular}{|c|c|}
    \cline{1-2}
    $i_1$ & $j_1$ \\
    \cline{1-2}
    $i_2$ \\
    \cline{1-1}
    $i_3$ \\
    \cline{1-1}
    $\vdots$ \\
    \cline{1-1}
    $i_s$ \\
    \cline{1-1}
\end{tabular}.
\]
\end{definition}

\begin{example}
Fix the block diagonal matching field $(1|2345678)$. For each $1 \le k \le 7$, the set $\Tc^k_{\ell}$ is included in $\Tc_{\ell}$. These are the tableaux of type 1. So we observe that $T_1$ is a tableau of type 1.B(3) shown below. The tableaux of type 2 in $\Tc_{\ell}$ are built from those of type 1. For example $T_2$ is a tableau of type 2.3B(3) obtained by adding $7,8$ to the left column of $T_1$. 
\[
T_1 = 
\begin{tabular}{|c|c|}
    \hline
    4 & 2 \\ \hline
    1 & 3 \\ \hline
    5 & 4 \\ \hline
    6 & 7 \\
    \hline
\end{tabular}\, ,
\quad T_2 = 
\begin{tabular}{|c|c|}
    \hline
    4 & 2 \\ \hline
    1 & 3 \\ \hline
    5 & 4 \\ \hline
    6 & 7 \\ \hline
    \multicolumn{1}{|c|}{7} \\ \cline{1-1}
    \multicolumn{1}{|c|}{8} \\
    \cline{1-1}
\end{tabular}\, ,
\quad T_3 =
\begin{tabular}{|c|c|}
    \hline
    2 & 1 \\ \hline
    \multicolumn{1}{|c|}{1} \\ \cline{1-1}
    \multicolumn{1}{|c|}{5} \\ \cline{1-1}
    \multicolumn{1}{|c|}{6} \\
    \cline{1-1}
\end{tabular}.
\]

The tableaux of type 3 are those which have exactly one column of size one. The tableau of type 3A are exactly those in  semi-standard form. For the other types, 3B, 3C and 3D, it is not possible to put these tableau into semi-standard form. For example $T_3$ is a tableau of type 3D and we cannot swap the entries in the first row.
\end{example}

\begin{remark}
Notice that the algebra $A_{\ell}$ has the standard grading and any monomial $P_{I_1}P_{I_2} \dots P_{I_t}$ in $A_{\ell}$ is identified with the tableau $T_{I_1 I_2 \dots I_t}$. Moreover, two monomials are equal in $A_{\ell}$ if and only if their corresponding tableaux are row-wise equal.  
\end{remark}

\begin{lemma}\label{lem:basis_fl_span}
The tableaux $\Tc_{\ell}$ span $[A_{\ell}]_2$.
\end{lemma}

\begin{proof}
We prove the lemma by showing that any tableau $T_{IJ}$ is row-wise equal to some tableau $T_{I'J'} $ in $ \Tc_{\ell}$. So fix $I, J \subseteq [n]$, we take cases on $|I|$ and $|J|$.

\textbf{Case 1.} Let $|I| = |J| = k$ for some $k$. By Theorem~\ref{thm:basis_gr} we have that there exists $I', J' \subseteq [n]$ such that $T_{IJ}$ and $T_{I'J'}$ are quadratically equivalent and $T_{I'J'} $ in $ \Tc^k_{\ell} \subset \Tc_{\ell}$.

\smallskip

\textbf{Case 2.} Let $|I| > |J| \ge 2$. Let $s = |I|$ and $t = |J|$. Consider the sub-tableau $T'$ of $T_{IJ}$ consisting of the first $t$ rows. By Theorem~\ref{thm:basis_gr} there exists $I', J' \subset [n]$ such that $T'$ is row-wise equal to $T_{I'J'}$ and $T_{I'J'} $ in $ \Tc^t_{\ell}$. Now if $i'_s \le j'_s$ set $I' := I' \cup \{i_{t+1}, \dots, i_s \}$. Otherwise if $i'_s > j'_s$ then set $J' := J' \cup \{i_{t+1}, \dots, i_s \}$. Then we have $T_{IJ}$ and $T_{I'J'}$ are row-wise equal and $T_{I'J'} $ in $ \Tc_{\ell}$ is a tableau of type 2. 

\smallskip

\textbf{Case 3.} Let $|I| \ge 2$ and $|J| = 1$. Now we take cases on $i_1, i_2$ and $j_1$. For each case $T_{IJ}$ is depicted in Figure~\ref{fig:lem_span_tableaux}. For simplicity we depict only those entries in the first two rows of $T_{IJ}$.

\smallskip

\textbf{Case 3.i.} Let $i_1, i_2, j_1 \in B_1$. If $i_1 \le j_1$ then we have that $T_{IJ} \in \Tc_{\ell}$ is a tableau of type 3A(1). On the other hand, if $i_1 > j_1$ then we may swap $i_1$ and $j_1$ to obtain the tableau $T_{I'J'} \in \Tc_{\ell}$ of type 3A(1).

\smallskip

\textbf{Case 3.ii.} Let $i_1, i_2 \in B_1, j_1 \in B_2$.
So $T_{IJ} \in \Tc_{\ell}$ is a tableau of type 3A(1).

\smallskip

\textbf{Case 3.iii.} Let $i_1, j_1 \in B_1, i_2 \in B_2$. If $i_1 \le j_1$ then we have that $T_{IJ} \in \Tc_{\ell}$ is a tableau of type 3D. If on the other hand $i_1 > j_1$ then we may swap $i_2$ and $j_1$ to obtain the tableau $T_{I'J'} \in \Tc_{\ell}$ of type 3A(1). 

\smallskip

\textbf{Case 3.iv.} Let $i_1 \in B_1, i_2, j_1 \in B_2$. 
If $i_2 \le j_1$ then we have that $T_{IJ} \in \Tc_{\ell}$ is a tableau of type 3C. If on the other hand $i_2 > j_1$ then we may swap $i_2$ and $j_1$ to obtain the tableau $T_{I'J'} \in \Tc_{\ell}$ of type 3C. 

\smallskip

\textbf{Case 3.v.} Let $j_1 \in B_1, i_1, i_2 \in B_2$.
So $T_{IJ} \in \Tc_{\ell}$ is a tableau of type 3B.

\smallskip

\textbf{Case 3.vi.} Let $i_1, i_2, j_1 \in B_2$.
If $i_1 \le j_1$ then we have that $T_{IJ} \in \Tc_{\ell}$ is a tableau of type 3A(2). If on the other hand $i_1 > j_1$ then we may swap $i_1$ and $j_1$ to obtain the tableau $T_{I'J'} \in \Tc_{\ell}$ of type 3A(2).

\begin{figure}
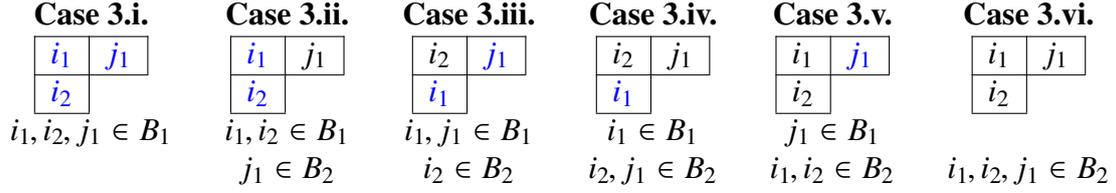

    \centering
    \begin{tabular}{c}
        \textbf{Case 3.i.} \\
        \begin{tabular}{|c|c|}
            \cline{1-2}
            \textcolor{blue}{$i_1$} & \textcolor{blue}{$j_1$} \\
            \cline{1-2}
            \textcolor{blue}{$i_2$} \\
            \cline{1-1}
        \end{tabular} \\
        $i_1, i_2, j_1 \in B_1$ \\
        \,
    \end{tabular}
    \,
    \begin{tabular}{c}
        \textbf{Case 3.ii.} \\
        \begin{tabular}{|c|c|}
            \cline{1-2}
            \textcolor{blue}{$i_1$} & $j_1$ \\
            \cline{1-2}
            \textcolor{blue}{$i_2$} \\
            \cline{1-1}
        \end{tabular} \\
        $i_1, i_2 \in B_1$ \\
        $j_1 \in B_2$
    \end{tabular}
    \,
    \begin{tabular}{c}
        \textbf{Case 3.iii.} \\
        \begin{tabular}{|c|c|}
            \cline{1-2}
            $i_2$ & \textcolor{blue}{$j_1$} \\
            \cline{1-2}
            \textcolor{blue}{$i_1$} \\
            \cline{1-1}
        \end{tabular}  \\
        $i_1, j_1 \in B_1$ \\
        $i_2 \in B_2$
    \end{tabular}
    \,
    \begin{tabular}{c}
        \textbf{Case 3.iv.} \\
        \begin{tabular}{|c|c|}
            \cline{1-2}
            $i_2$ & $j_1$ \\
            \cline{1-2}
            \textcolor{blue}{$i_1$} \\
            \cline{1-1}
        \end{tabular}  \\
        $i_1 \in B_1$ \\
        $i_2, j_1 \in B_2$
    \end{tabular}
    \,
    \begin{tabular}{c}
        \textbf{Case 3.v.} \\
        \begin{tabular}{|c|c|}
            \cline{1-2}
            $i_1$ & \textcolor{blue}{$j_1$} \\
            \cline{1-2}
            $i_2$ \\
            \cline{1-1}
        \end{tabular}  \\
        $j_1 \in B_1$ \\
        $i_1, i_2 \in B_2$
    \end{tabular}
    \,
    \begin{tabular}{c}
        \textbf{Case 3.vi.} \\
        \begin{tabular}{|c|c|}
            \cline{1-2}
            $i_1$ & $j_1$ \\
            \cline{1-2}
            $i_2$ \\
            \cline{1-1}
        \end{tabular}  \\
        \, \\
        $i_1, i_2, j_1 \in B_2$
    \end{tabular}
    \caption{Depiction of the first two rows of tableaux in Case 3 of the proof of Lemma~\ref{lem:basis_fl_span}.
    }
    \label{fig:lem_span_tableaux}
\end{figure}
\end{proof}

\begin{example}\label{example:basis_fl_span}
For each case in Lemma~\ref{lem:basis_fl_span} we provide an example of the manipulation of the tableaux. For Cases 1 and 2, the manipulations follow from Theorem~\ref{thm:basis_gr} so we consider Case 3 and fix the block diagonal matching field $(123|456)$ for $\Flag_6$.

\textbf{Case 3.i.}
Consider the tableau $T$ below. In this case we swap the entries in the first row to obtain $T'$ which is a tableau of type 3A(1).
\[
T = 
\begin{tabular}{|c|c|}
    \cline{1-2}
    $\textbf{\textcolor{blue}{2}}$ & $\textbf{\textcolor{blue}{1}}$ \\
    \cline{1-2}
    $3$ \\
    \cline{1-1}
    $4$ \\
    \cline{1-1}
\end{tabular}\, ,
\quad T' =
\begin{tabular}{|c|c|}
    \cline{1-2}
    $1$ & $2$ \\
    \cline{1-2}
    $3$ \\
    \cline{1-1}
    $4$ \\
    \cline{1-1}
\end{tabular}.
\]

\textbf{Case 3.ii.}
The tableau $T$ below is an example of a tableau of type 3A(1) which arises in this case.
\[
T = 
\begin{tabular}{|c|c|}
    \cline{1-2}
    $1$ & $4$ \\
    \cline{1-2}
    $3$ \\
    \cline{1-1}
    $4$ \\
    \cline{1-1}
\end{tabular}.
\]

\textbf{Case 3.iii.}
Consider the tableau $T$ below. The entry of the right column is less that the second entry of the left column. So we may swap the entries in the first row to obtain $T'$ which is a tableau of type 3A(1). On the other hand, the entry in the right column of $T''$ is greater or equal to the second row entry in the left column. So $T''$ is a tableau of type 3D.
\[
T = 
\begin{tabular}{|c|c|}
    \cline{1-2}
    $\textbf{\textcolor{blue}{4}}$ & $\textbf{\textcolor{blue}{1}}$ \\
    \cline{1-2}
    $2$ \\
    \cline{1-1}
    $5$ \\
    \cline{1-1}
\end{tabular}\, ,
\quad T' =
\begin{tabular}{|c|c|}
    \cline{1-2}
    $1$ & $4$ \\
    \cline{1-2}
    $2$ \\
    \cline{1-1}
    $5$ \\
    \cline{1-1}
\end{tabular}\, ,
\quad T'' =
\begin{tabular}{|c|c|}
    \cline{1-2}
    $4$ & $2$ \\
    \cline{1-2}
    $1$ \\
    \cline{1-1}
    $5$ \\
    \cline{1-1}
\end{tabular}.
\]

\textbf{Case 3.iv.} Consider the tableau $T$ below. We may swap the entries in the first row to obtain $T'$, a tableau of type 3C.

\[
T = 
\begin{tabular}{|c|c|}
    \cline{1-2}
    $\textbf{\textcolor{blue}{5}}$ & $\textbf{\textcolor{blue}{4}}$ \\
    \cline{1-2}
    $1$ \\
    \cline{1-1}
    $6$ \\
    \cline{1-1}
\end{tabular}\, ,
\quad T' =
\begin{tabular}{|c|c|}
    \cline{1-2}
    $4$ & $5$ \\
    \cline{1-2}
    $1$ \\
    \cline{1-1}
    $6$ \\
    \cline{1-1}
\end{tabular}.
\]

\textbf{Case 3.v.} The tableau $T$ below is an example of a tableau of type 3B:
\[
T = 
\begin{tabular}{|c|c|}
    \cline{1-2}
    $4$ & $1$ \\
    \cline{1-2}
    $5$ \\
    \cline{1-1}
    $6$ \\
    \cline{1-1}
\end{tabular}.
\]

\textbf{Case 3.vi.} Consider the tableau $T$ below. Since all its entries lie in $B_2$ we can put it in semi-standard form by a swap. The resulting tableau $T'$ is of type 3A(2).

\[
T = 
\begin{tabular}{|c|c|}
    \cline{1-2}
    $\textbf{\textcolor{blue}{5}}$ & $\textbf{\textcolor{blue}{4}}$ \\
    \cline{1-2}
    $6$ \\
    \cline{1-1}
\end{tabular}\, ,
\quad T' =
\begin{tabular}{|c|c|}
    \cline{1-2}
    $4$ & $5$ \\
    \cline{1-2}
    $6$ \\
    \cline{1-1}
\end{tabular}.
\]
\end{example}

\begin{lemma}\label{lem:basis_fl_indep}
The tableaux $\Tc_{\ell}$ are linearly independent.
\end{lemma}

\begin{proof}
Since the ideal associated to the block diagonal matching field is generated by binomials, it suffices to show that no two tableaux in $\Tc_{\ell}$ are row-wise equal. Suppose by contradiction that $T_{IJ}$ and $T_{I'J'}$ are two distinct row-wise equal tableaux in $\Tc_{\ell}$. We may assume without loss of generality $|I| = |I'|$ and $|J| = |J'|$. We now proceed by taking cases on the type of $T_{IJ}$.

\smallskip

\textbf{Case 1.} Let $T_{IJ}$ be of type 1. Then $|I| = |I'| = |J|  = |J'| = k$ for some $k$. Therefore $T_{I'J'}$ is also of type 1 and $T_{IJ}$ and $T_{I'J'}$ belong to $\Tc_{\ell}^k$. However $T_{IJ}$ and $T_{I'J'}$ are linearly independent by Theorem~\ref{thm:basis_gr}, a contradiction.

\smallskip

\textbf{Case 2.} Let $T_{IJ}$ be of type 2. Then $|I| > |J| \ge 2$. Let $U_{IJ}$ be the sub-tableau of $T_{IJ}$ consisting of the first $|J|$ rows and $V_{I}$ be the sub-tableau consisting of the final $|I| - |J|$ rows. Note that $V_I$ depends only on the contents of $I$. Similarly we define $U_{I'J'}$ and $V_{I'}$. By row-wise equality of $T_{IJ}$ and $T_{I'J'}$, it follows that $V_{I}$ and $V_{I'}$ are equal and so $U_{IJ}$ and $U_{I'J'}$ are row-wise equal and distinct. By Definition~\ref{def:basis_gr}, we have that $U_{IJ}$ and $U_{I'J'}$ belong to $\Tc_{\ell}^{|J|}$. However, since $U_{IJ}$ and $U_{I'J'}$ are distinct by Theorem~\ref{thm:basis_gr} they are linearly independent, a contradiction.

\smallskip

\textbf{Case 3.} Let $T_{IJ}$ be of type 3. So $|I| > |J| = 1$. Note that, by row-wise equality, $T_{I'J'}$ is obtained from $T_{IJ}$ by swapping the entries in the first row. In particular, given $T_{IJ}$ there is exactly one option for $T_{I'J'}$. We now take cases on the sub-types of $T_{IJ}$, see Definition~\ref{def:basis_fl}.  Note that when considering the type of $T_{IJ}$ and $T_{I'J'}$, we need only consider the first two rows of the tableaux.

\smallskip

\textbf{Case 3.A.} Let $T_{IJ}$ be of type 3.A. So the tableaux have the following form,
\[
T_{IJ} = 
\begin{tabular}{|c|c|}
    \cline{1-2}
    $i_1$ & $j_1$  \\
    \cline{1-2}
    $i_2$ \\
    \cline{1-1}
\end{tabular}\, , \quad
T_{I'J'} = 
\begin{tabular}{|c|c|}
    \cline{1-2}
    $j_1$ & $i_1$ \\
    \cline{1-2}
    $i_2$ \\
    \cline{1-1}
\end{tabular} \,
\]
where $i_1 < i_2$ and $i_1 \le j_1$. If $T_{I'J'}$ is also of type 3.A then we have $i_1 = j_1$ and so $T_{IJ}$ and $T_{I'J'}$ are equal, a contradiction. Otherwise, 
if $T_{I'J'}$ is not of type 3.A, then we must have that $i_1, i_2 \in B_1$ and $j_1 \in B_2$. Therefore $T_{I'J'}$ is a tableau of type 3.D and $i_2 \le i_1$, a contradiction.

\smallskip

\textbf{Case 3.B.}
Let $T_{IJ}$ be of type 3.B. So the tableaux have the following form,
\[
T_{IJ} = 
\begin{tabular}{|c|c|}
    \cline{1-2}
    $i_1$ & \textcolor{blue}{$j_1$}  \\
    \cline{1-2}
    $i_2$ \\
    \cline{1-1}
\end{tabular}\, , \quad
T_{I'J'} = 
\begin{tabular}{|c|c|}
    \cline{1-2}
    \textcolor{blue}{$j_1$} & $i_1$ \\
    \cline{1-2}
    $i_2$ \\
    \cline{1-1}
\end{tabular} \,
\]
where $i_1, i_2 \in B_2$, $j_1 \in B_1$ and $i_1 < i_2$. However, $T_{I'J'}$ is not a valid matching field tableau, a contradiction.

\smallskip

\textbf{Case 3.C.}
Let $T_{IJ}$ be of type 3.C. So the tableaux have the following form,
\[
T_{IJ} = 
\begin{tabular}{|c|c|}
    \cline{1-2}
    $i_2$ & $j_1$  \\
    \cline{1-2}
    \textcolor{blue}{$i_2$} \\
    \cline{1-1}
\end{tabular}\, , \quad
T_{I'J'} = 
\begin{tabular}{|c|c|}
    \cline{1-2}
    $j_1$ & $i_2$ \\
    \cline{1-2}
    \textcolor{blue}{$i_1$} \\
    \cline{1-1}
\end{tabular} \,
\]
where $i_1 \in B_1$, $ i_2, j_1 \in B_2$ and $i_2 \le j_1$. It follows that $T_{I'J'}$ is also a tableau of type 3.C and so $i_2 = j_1$. Therefore $T_{IJ}$ and $T_{I'J'}$ are equal, a contradiction.

\smallskip

\textbf{Case 3.D.}
Let $T_{IJ}$ be of type 3.D. So the tableaux have the following form,
\[
T_{IJ} = 
\begin{tabular}{|c|c|}
    \cline{1-2}
    $i_2$ & \textcolor{blue}{$j_1$}  \\
    \cline{1-2}
    \textcolor{blue}{$i_2$} \\
    \cline{1-1}
\end{tabular}\, , \quad
T_{I'J'} = 
\begin{tabular}{|c|c|}
    \cline{1-2}
    \textcolor{blue}{$j_1$} & $i_2$ \\
    \cline{1-2}
    \textcolor{blue}{$i_1$} \\
    \cline{1-1}
\end{tabular} \,
\]
where $i_1, j_1 \in B_1$, $ i_2 \in B_2$ and $i_1 \le j_1$. However $T_{I'J'}$ is not a valid matching field tableau, a contradiction.
\end{proof}

\begin{lemma}\label{lem:basis_fl_biject}
The map $S$ is a bijection between $\Tc_{\ell}$ and the set of semi-standard tableaux with two columns.
\end{lemma}

\begin{proof}
We proceed to show that the inverse to $S$ is well defined. Let $T = T_{IJ}$ be a tableau in semi-standard form. We will take cases on $|I|$ and $|J|$.

\textbf{Case 1.} Let $|I| = |J| = k$, for some $k$. By Theorem~\ref{thm:basis_gr} we have $S^{-1}(T_{IJ}) = (S^k)^{-1}(T_{IJ})$.

\smallskip

\textbf{Case 2.} Let $|I| > |J| \ge 2$ and $t = |J|$ for some $t$. Let $\tilde{T}$ be the sub-tableau of $T$ consisting of the first $t$ rows. By Theorem~\ref{thm:basis_gr}, we have that $S^t$ is a bijection between $T^t_{\ell}$ and rectangular semi-standard tableau with two columns and $t$ rows. Let $T_{I'J'} = (S^t)^{-1}(\tilde{T})$. If $i'_t \le j'_t$ then redefine $I' := I' \cup \{ i_{t+1}, \dots, i_{s} \}$. Otherwise, redefine $J' := J' \cup \{i_{t+1}, \dots, i_s \}$. So $S^{-1}(T_{IJ}) = T'_{I'J'}$.

\smallskip

\textbf{Case 3.} Let $|I| \ge 2$ and $ |J| = 1$. We take cases on $i_1, i_2, j_1$. 

\textbf{Case 3.i.} Let $i_1, i_2 \in B_1$. We have $S^{-1}(T) = T$ which is a tableau of type 3A(1).

\smallskip

\textbf{Case 3.ii.} Let $i_1, j_1 \in B_1$ and $ i_2 \in B_2$. Then $S^{-1}(T) = T'$ is the tableau of type 3D from Page~\pageref{page:3D} as shown in Figure~\ref{fig:lem_bij_tableaux} (left).

\smallskip

\textbf{Case 3.iii.} Let $i_1 \in B_1$ and $ i_2, j_1 \in B_2$. We take cases on $j_1$ and $i_2$.

\textbf{Case 3.iii.a.} Let $j_1 < i_2$. Then $S^{-1}(T) = T'$ is the tableau of type 3B. See Figure~\ref{fig:lem_bij_tableaux} (middle).

\textbf{Case 3.iii.b.} Let $j_1 \ge i_2$. Then $S^{-1}(T) = T'$ is the tableau of type 3C. See Figure~\ref{fig:lem_bij_tableaux} (right).
\smallskip

\textbf{Case 3.iv.} Let $i_1, i_2, j_1 \in B_2$. We have $S^{-1}(T) = T$ which is a tableau of type 3A(2). 

\begin{figure}
    \centering
    \begin{tabular}{c}
        \textbf{Case 3.ii.} \\
        \begin{tabular}{|c|c|}
            \cline{1-2}
            $i_2$ & $j_1$ \\
            \cline{1-2}
            $i_1$ \\
            \cline{1-1}
            $i_3$ \\
            \cline{1-1}
            $\vdots$ \\
            \cline{1-1}
            $i_s$ \\
            \cline{1-1}
        \end{tabular} \\
        $i_1, j_1 \in B_1$ \\
        $i_2 \in B_2$ \\
        \,
    \end{tabular}
    \quad
    \begin{tabular}{c}
        \textbf{Case 3.iii.a.} \\
        \begin{tabular}{|c|c|}
            \cline{1-2}
            $j_1$ & $i_1$ \\
            \cline{1-2}
            $i_2$ \\
            \cline{1-1}
            $i_3$ \\
            \cline{1-1}
            $\vdots$ \\
            \cline{1-1}
            $i_s$ \\
            \cline{1-1}
        \end{tabular} \\
        $i_1 \in B_1$ \\
        $i_2, j_1 \in B_2$ \\
        $j_1 < i_2$
    \end{tabular}
    \quad
    \begin{tabular}{c}
        \textbf{Case 3.iii.b.} \\
        \begin{tabular}{|c|c|}
            \cline{1-2}
            $i_2$ & $j_1$ \\
            \cline{1-2}
            $i_1$ \\
            \cline{1-1}
            $i_3$ \\
            \cline{1-1}
            $\vdots$ \\
            \cline{1-1}
            $i_s$ \\
            \cline{1-1}
        \end{tabular} \\
        $i_1 \in B_1$ \\
        $i_2, j_1 \in B_2$ \\
        $j_1 \ge i_2$
    \end{tabular}
    
    \caption{The tableaux for Case 3 in the proof of Lemma~\ref{lem:basis_fl_biject}}
    \label{fig:lem_bij_tableaux}
\end{figure}

In each case the pre-image of $S$ is unique, so we have shown that $S$ is a bijection from $\Tc_{\ell}$ to semi-standard tableaux with two columns.
\end{proof}

\begin{example}
Fix the block diagonal matching field $B_3=(123|456)$. Consider the semi-standard tableau $T$ below. Since the right hand column $J$ has size one, we look at the values in the first two rows. The values are $i_1 = 1, i_2 = 4$ and $j_1 = 2$, so we are in Case 3.ii. of the proof of Lemma~\ref{lem:basis_fl_biject}. Hence $S(T') = T$ for the following tableau $T' \in \Tc_{3}$,
\[
T = 
        \begin{tabular}{|c|c|}
            \multicolumn{1}{c}{$I$} & \multicolumn{1}{c}{$J$} \\
            \cline{1-2}
            1 & 2 \\
            \cline{1-2}
            4 \\
            \cline{1-1}
            6 \\
            \cline{1-1}
        \end{tabular}\, ,
\quad T' = 
        \begin{tabular}{|c|c|}
            \multicolumn{1}{c}{$I'$} & \multicolumn{1}{c}{$J'$} \\
            \cline{1-2}
            4 & 2 \\
            \cline{1-2}
            1 \\
            \cline{1-1}
            6 \\
            \cline{1-1}
        \end{tabular}.
\]
In particular, the tableau $T'$ lies in $\Tc_{3}$ and is of type 3D. Note that this manipulation of the tableau depends only upon the first two rows. We give examples for all manipulations from Case 3 of Lemma~\ref{lem:basis_fl_biject} in Figure~\ref{fig:bij_examples}, where we depict the changes only in the first two rows.
\end{example}

\begin{figure}
    \centering
    \includegraphics[scale=0.75]{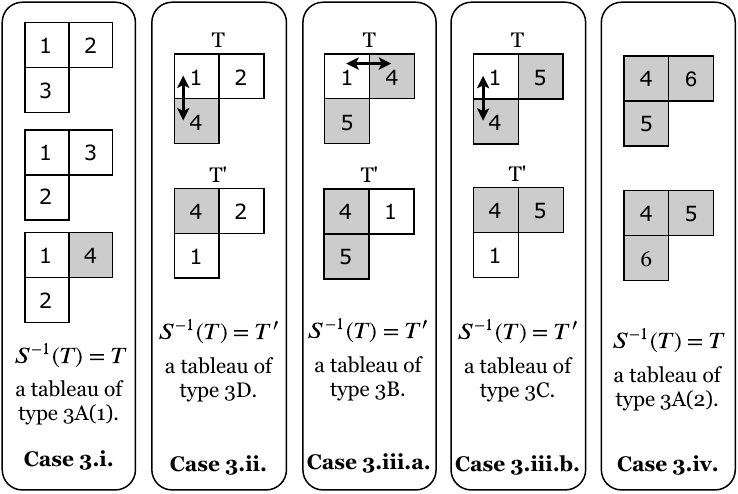}
    \caption{Examples of semi-standard tableau and their pre-image under the map $S$. The cells of the tableau are colored so that entries in $B_1 = \{1,2,3 \}$ are white and the entries in $B_2 = \{4,5,6 \}$ are shaded.}
    \label{fig:bij_examples}
\end{figure}

Now, we are ready to present our main result which extends the classical result from the diagonal matching field tableaux, known as semi-standard tableaux to block diagonal matching field tableaux.
We follow the notation in Definition~\ref{sagbi}.

\begin{theorem}\label{prop:SAGBI}
The Pl\"ucker forms are a SAGBI basis for the Pl\"ucker algebra with respect to the weight vectors arising from block diagonal matching fields.
\end{theorem}

\begin{proof}
By Theorem~\ref{prop:quad}, we have seen that the matching field ideal $J_{B_{\ell}}$ is quadratically generated. In order to show that $\{P_I:\ \emptyset\neq I\subsetneq [n]\}$ forms a SAGBI basis for $\mathcal{A}_{n}$, by \cite[Theorem~11.4]{sturmfels1996grobner} it is enough to show that $\init_{w_\ell}(I_n) = J_{\BLambda}$.

By Lemma~\ref{lem:basis_fl_span} and Lemma~\ref{lem:basis_fl_indep} we have that $\Tc_{\ell}$ is a basis for $[A_{\ell}]_2$. Then by Lemma~\ref{lem:basis_fl_biject} we have that $\Tc_{\ell}$ is in bijection with semi-standard tableaux with two columns. Since semi-standard tableaux form a basis for $A_{0}$, we have shown that $\dim([A_{\ell}]_2) = \dim([A_{0}]_2)$.

To show that $\init_{w_\ell}(I_n) = J_{\BLambda}$, we note that $J_{\BLambda}$ is quadratically generated and $\dim([A_{\ell}]_2) = \dim([A_{0}]_2)$. Therefore we may apply the same method as the proof in \cite[Theorem~4.3]{OllieFatemeh} because \cite[Lemma~4.4, 4.5 and 4.6]{OllieFatemeh} also hold for $I_n$.
\end{proof}
As a corollary of the above statements and \cite[Theorem 11.4]{sturmfels1996grobner} we have that:

\begin{corollary} \label{cor:SAGBI}
Each block diagonal matching field produces a toric degeneration of $\Flag_n$. 
\end{corollary}

\begin{table}
    \centering
    \footnotesize
      \resizebox{7cm}{!}{\begin{tabular}{|c|c|c|c|}
        \hline
       $\Flag_4$ & ${\ell}$ & The $f$-vector of the toric polytope \\
        \hline
       \hline
        \multirow{3}{*}{ } & 0 & 40, 132, 186, 139, 57, 12    \\
        \cline{2-3}
        & 1 & 42, 141, 202, 153, 63, 13    \\
        \cline{2-3}
        & 2,3 & 43, 146, 212, 163, 68, 14  \\
        \hline
    \end{tabular}
    }
    \caption{For $\Flag_4$ and each matching field $B_{\ell}$ we calculate the f-vector of its associated toric polytope.}
    \label{tab:toric_polytopes}
\end{table}

\begin{remark}
Using the combinatorial tools of matching field tableaux we have provided a family of toric degenerations of opposite Schubert varieties and Richardson varieties in \cite{OllieFatemeh3,Ollie4}. We remark that all such toric degenerations can be realized as Gr\"obner degenerations, nevertheless, this is not true in general; See e.g.
\cite{kateri2015family}.
\end{remark}

\begin{remark}
Let $B_\ell$ be a block diagonal matching field. We have seen that $B_{\ell}$ produces a toric degeneration of $\Flag_n$. So we define the \emph{matching field polytope} to be the polytope of the associated toric variety.
Table~\ref{tab:toric_polytopes} shows computational results for matching field polytopes for $\Flag_4$. In particular, we see that the matching fields give distinct toric ideals except for $\ell = 2$ and $\ell = 3$, which are unimodular equivalent. We note that here the polytopes for $\ell = 0$ and $\ell = 1$ are the Gelfand-Cetlin polytope and the FFLV polytope respectively.
In \cite{Akihiro} we
study these polytopes from a geometric point of view.
\end{remark}

\bibliographystyle{alpha} 
\bibliography{Trop1.bib}






\end{document}